\documentclass[10pt]{amsart}

\usepackage{amsmath, amssymb, amsthm}
\usepackage{graphicx}
\usepackage{multirow}
\usepackage{caption}
\usepackage{tabularx}
\usepackage{amsmath}
\usepackage{amssymb}
\usepackage{mathrsfs}
\usepackage{hyperref}  

\newcommand{\cR}{\mathcal{R}}
\newcommand{\cS}{\mathcal{S}}
\newcommand{\cT}{\mathcal{T}}
\newcommand{\cV}{\mathcal{V}}
\newcommand{\cW}{\mathcal{W}}

\theoremstyle{plain}
\newtheorem{theorem}{Theorem}[section]
\newtheorem{lemma}[theorem]{Lemma}

\theoremstyle{definition}

\theoremstyle{remark}

\makeatletter
\renewcommand{\section}{%
  \@startsection{section}{1}%
    \z@{.7\linespacing\@plus.7\linespacing}{.5\linespacing}%
    {\centering\normalfont\Large\bfseries}}

\def\subsection{\@startsection{subsection}{2}%
  \z@{.7\linespacing\@plus.7\linespacing}{.5\linespacing}%
  {\normalfont\normalsize\bfseries}} 
\makeatother

\begin{document}

\title[Undecidability of translational tiling with 2 polycubes]
{Undecidability of translational tiling with \\ 2 polycubes}

\author{Yoonhu Kim}
\address{\parbox[t]{\textwidth} {Yoonhu Kim \\ Gyeonggi Science High School \\ Gyeonggi-do 16297, Korea} }
\email{rumstis@gmail.com}

\subjclass[2020]{Primary 52C22, 68Q17}
\keywords{tiling, translation, polycube, undecidability}

\begin{abstract}
In this paper, we prove that it is undecidable whether a set of two polycubes can tile $\mathbb{Z}^3$ by translation.
The proof involves a new technique that allows us to simulate two disconnected polycubes with two connected polycubes.
By expanding this technique to higher dimensions, we also prove that a set of disconnected tiles in $\mathbb{Z}^n$
can be simulated by the same number of connected tiles in $\mathbb{Z}^n$ for $n \geq 3$.

\end{abstract}

\maketitle

\section{Introduction} 

\subsection{Tiling and undecidability} 

Problems on the decidability of tilings have constantly attracted many mathematicians.
One of the earliest studies was by Hao Wang in $1961$, where the concept of Wang tiles was introduced~\cite{HW}.
Wang tiles are unit squares with a color assigned to each edge.
Wang’s domino problem asks whether it is possible to tile the plane with a finite set of Wang tiles, with the constraint that matching edges must have the same color.
This problem has later been proven undecidable by Berger, who showed that it was possible to encode any Turing machine using Wang tiles~\cite{RB}.

A translational tiling problem in $\mathbb{Z}^n$ asks whether a finite set of tiles in $\mathbb{Z}^n$ can cover $\mathbb{Z}^n$ using translational copies.
In 1970, Golomb showed that any set of Wang tiles can be encoded by the same number of tiles in $\mathbb{Z}^2$.
Thus, the translational tiling problem in $\mathbb{Z}^2$ is undecidable if the number of tiles can be arbitrarily large~\cite{SW}.
Ollinger proposed the question of tiling $\mathbb{Z}^2$ with $k$ tiles, where $k$ is a fixed integer~\cite{NO}.
In the same paper, Ollinger showed that this problem is undecidable for $k=11$ by constructing $11$ tiles that simulate a given set of Wang tiles.
In other words, the translational tiling problem is undecidable for $(n, \, k)=(2, \, 11)$.

There have been various works proving decidability or undecidability for certain parameters $(n, \, k)$.
For $n=1$ and $(n, \, k)=(2, \, 1)$, the problem is known to be decidable~\cite{VS, HA}.
The currently best undecidability results are $(2, \, 4)$ from Stade~\cite{JS} and $(4, \, 3)$ from Yang and Zhang~\cite{CY}.
Trivially, this implies that the problem is undecidable if $2 \leq n \leq3 , \, k \geq 4$ or $n \geq 4, \, k \geq3$.

Although some papers have studied the translational tiling problem for disconnected tiles, the above two results on undecidability utilize connected tiles.
We will also assume that the problem requires the tiles to be connected.

In this paper, we prove that the problem is undecidable for $(n, \, k)=(3, \, 2)$.\\
In Section~2, we develop a technique that allows us to simulate disconnected polycubes using connected polycubes.\\
In Section~3, we propose a problem called the cyclic triomino problem and prove its undecidability.\\
In Section~4, we construct a set of two polycubes that can encode an arbitrary cyclic triomino problem to prove our main result.\\
In Section~5, we generalize the results of Section~2 to higher dimensions for wider future implications.

\subsection{Notation} 

In this paper, we use three types of operations on sets of position vectors:

\[
\begin{aligned}
A \oplus B &= \{\, a+b \mid a \in A, \, b \in B \,\} \\
cA &= \{\, ca \mid a \in A \,\} \\
v+A &= A+v = \{\, v+a \mid a \in A \,\} 
\end{aligned}
\]
where $A$, $B$ are sets of position vectors, $v$ is a position vector and $c$ is a positive integer.
Note that all position vectors used in this paper are in $\mathbb{Z}^n$ (where $n$ is some positive integer).
In every usage of the Minkowski sum $A \oplus B$, $A$ and $B$ will be chosen so that the operation is non-overlapping;
that is, $a+b$ is unique for every combination of $a \in A$ and $b \in B$.

An $n$-dimensional tile(including its position) is a non-empty finite subset of $\mathbb{Z}^n$.
The terms \textbf{polyomino} and \textbf{polycube} are used to refer to a $2$-dimensional and $3$-dimensional tile respectively.
All tiles we discuss throughout this paper are subsets of $\mathbb{Z}^n$.

For an $n$-dimensional tile $P$ and a subset $E$ of $\mathbb{Z}^n$,
we say that $P$ tiles $E$ by $W$ if $W \oplus P$ is non-overlapping and $W \oplus P =E$.
For a set of $n$-dimensional tiles $\{ P_1 , \, P_2 , \, \dots , \, P_k \}$, we say that $P_1 , \, P_2 , \, \dots , \, P_k$ tiles $E$
if $E$ can be partitioned into  $E_1 , \, E_2 , \, \dots , \, E_s$ such that one of $P_1 , \, P_2 , \, \dots , \, P_k$ tiles $E_i$ for every $1 \leq i \leq s$.

For $n$-dimensional tiles $P$ and $Q$ satisfying $P \cap Q = \varnothing $,
we say that $P$ and $Q$ are adjacent if there exists $p \in P$, $q \in Q$ and a unit vector $v$ parallel to an axis such that $p+v=q$.

An $n$-dimensional tile $P$ is \textbf{disconnected} if $P$ can be partitioned into sets $A$ and $B$ which are not adjacent.
An $n$-dimensional tile is \textbf{connected} if it is not disconnected.

\section{Simulating disconnected polycubes with connected polycubes} 

In this section, we show that the tiling of a set of $k$ disconnected polycubes can be simulated by a set of $k$ connected polycubes.
We first show that it is possible to partition a cube so that the parts meet certain conditions.
Then, using this partition, we construct a set of $k$ connected polycubes simulating a given set of $k$ disconnected polycubes.

Throughout this section, we will use the notation $C_l = \{ 0,\, 1,\, \ldots ,\, l - 1 \}^3$ to refer to a $l \times l \times l$ cube.

\subsection{The partition} 

We first define some desirable properties of the partition that will be used.
Let $P_1,\, P_2, \, \dots ,\, P_s$ be a partition of $C_l$ such that $P_1,\, P_2,\, \dots ,\, P_s$ are all connected polycubes.
A partition is \textbf{internally adjacent} when:
\[
P_i \text{ and } P_j \text{ are adjacent for every } 1 \leq i < j \leq s.
\]
A partition is \textbf{externally adjacent} when:
\[
\begin{aligned}
P_i & + lv \text{ and } P_j \text{ are adjacent for every } 1 \leq i < j \leq s \\
& \text{and every unit vector } v \text{ parallel to an axis.}
\end{aligned}
\]

Since $P_1,\, P_2,\, \dots ,\, P_s$ are all contained inside an $l \times l \times l$ cube,
$P_i + lv$ and $P_j$ are adjacent if and only if $a + (l - 1)v = b$ for some $a \in P_i$ and $b \in P_j$.
This equivalent condition will be used later to show that a partition is externally adjacent.

Now we describe a specific partition of $C_{m+2}$.
We first define a function $f$ that takes a coordinate from $C_{m+2}$ and outputs which partition the coordinate belongs to.

\[
f(x,\, y,\, z) =
\begin{cases}
y & (0 \leq x \leq m,\, 1 \leq y \leq m,\, z = 0)\\
x & (1 \leq x \leq m,\, 0 \leq y \leq m,\, z = m+1)\\
y & (x = 0,\, 1 \leq y \leq m,\, 1 \leq z \leq m)\\
x & (1 \leq x \leq m,\, y = 0,\, 1 \leq z \leq m)\\
z & (1 \leq x \leq m+1,\, 1 \leq y \leq m+1,\, 1 \leq z \leq m)
\end{cases}
\]

For coordinates that do not belong in any of the five cases above, $f(x,\, y,\, z)$ could be assigned any single value.
We refer to these coordinates as \textbf{free coordinates}.
We will simply define $f(x, \, y, \, z) = 1$ if $(x, \, y, \, z)$ is a free coordinate.
Figure $1$ shows the value of $f(x, \, y, \, z)$ for the case $m = 4$; $z$ is increasing from left to right and the blank spaces represent free coordinates.

\begin{figure}[ht]
\centering
\includegraphics[width=0.95 \textwidth]{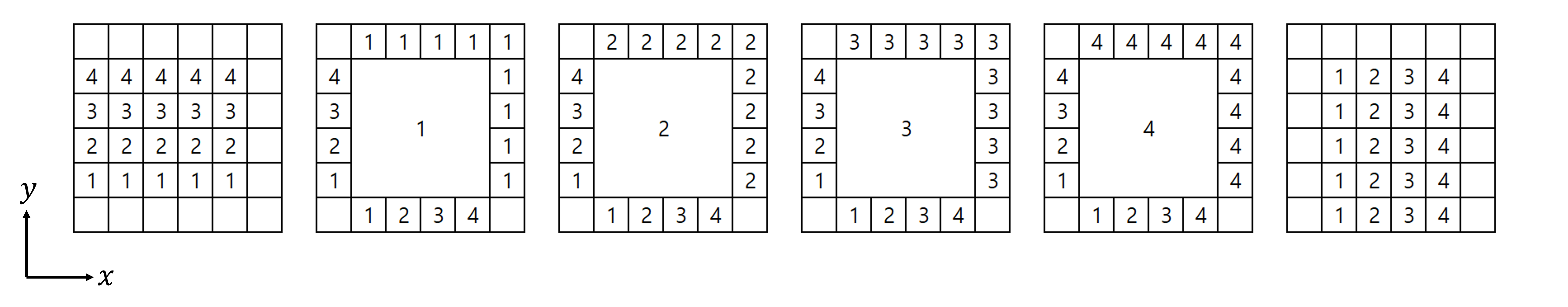}
\caption{The value of $f(x, \, y, \, z)$ for the case $m=4$}
\label{fig:Figure1}
\end{figure}

Let $Q_1, \, Q_2, \, \dots ,\, Q_m$ be a partition of $C_{m+2}$ such that
\[
(x,\, y,\, z) \in Q_i \iff f(x,\, y,\, z) = i.
\]

Figure $2$ shows $Q_1, \, Q_2, \, Q_3$ and $Q_4$ for the case $m=4$; free coordinates have been omitted.

\begin{figure}[ht]
\centering
\includegraphics[width=0.8 \textwidth]{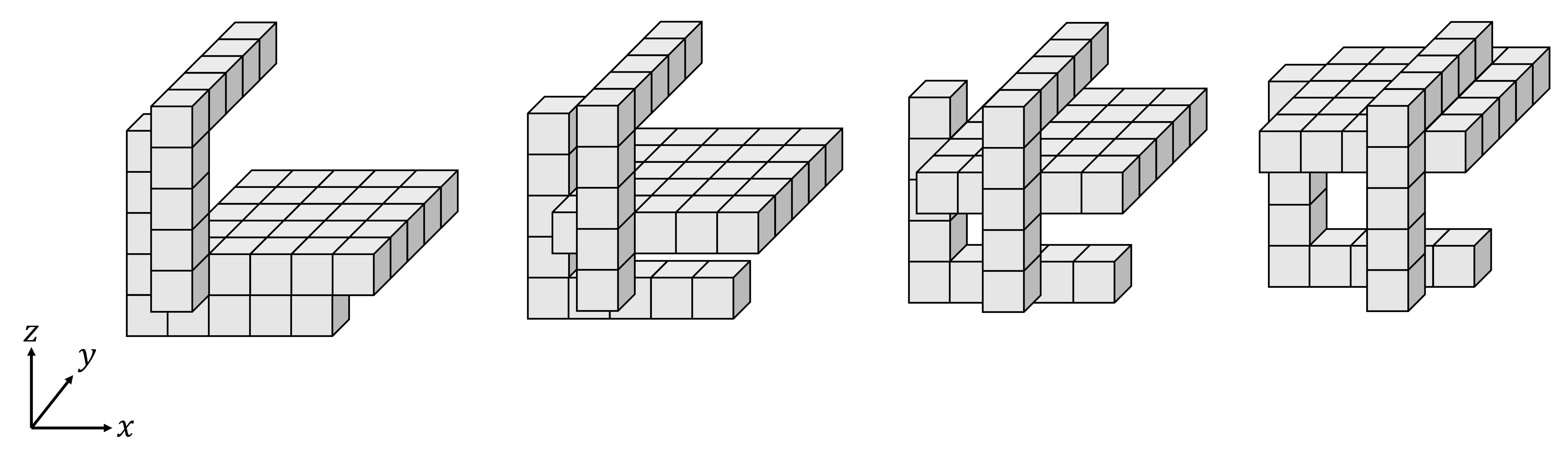}
\caption{$Q_1, \, Q_2 , \, Q_3$ and $Q_4$(from left to right) for $m=4$}
\label{fig:Figure2}
\end{figure}

It can easily be checked that $Q_1, \, Q_2, \, \dots ,\, Q_m$ are all connected polycubes.
Let us show that this partition is both internally adjacent and externally adjacent.

\begin{lemma} 
The partition $Q_1, \, Q_2, \, \dots , \, Q_m$ is internally adjacent.
\end{lemma}

\begin{proof}
For every $1 \leq i < j \leq m$, $f(i,\, 0,\, j) = i$ and $f(i,\, 1,\, j) = j$.
Thus $(i,\, 0,\, j) \in Q_i$ and $(i,\, 1,\, j) \in Q_j$.
Since these two coordinates differ by the unit vector $(0,\, 1,\, 0)$, $Q_i$ and $Q_j$ are adjacent.
\end{proof}

\begin{lemma} 
The partition $Q_1,\, Q_2,\, \dots ,\, Q_m$ is externally adjacent.
\end{lemma}

\begin{proof}
Let us show that for every $1 \leq i < j \leq m$ and every unit vector $v$ parallel to an axis,
there exists $a \in Q_i$ and $b \in Q_j$ such that $a + (m+1)v = b$.
\[
\begin{aligned}
v = (1, \, 0, \, 0)&: (0, \, i, \, j) \in Q_i, \, (m+1, \, i, \, j) \in Q_j \\
v = (0, \, 1, \, 0)&: (i, \, 0, \, j) \in Q_i, \, (i, \, m+1, \, j) \in Q_j \\
v = (0, \, 0, \, 1)&: (j, \, i, \, 0) \in Q_i, \, (j, \, i, \, m+1) \in Q_j \\
v = (-1, \, 0, \, 0)&: (m+1, \, j, \, i) \in Q_i, \, (0, \, j, \, i) \in Q_j \\
v = (0, \, -1, \, 0)&: (j, \, m+1, \, i) \in Q_i, \, (j, \, 0, \, i) \in Q_j \\
v = (0, \, 0, \, -1)&: (i, \, j, \, m+1) \in Q_i, \, (i, \, j, \, 0) \in Q_j
\end{aligned}
\]

The coordinates given above prove our claim for all six unit vectors $v$ parallel to an axis.
Therefore, the partition $Q_1,\, Q_2,\, \dots ,\, Q_m$ is externally adjacent.
\end{proof}

A small modification needs to be made to the partition $Q_1,\, Q_2,\, \dots ,\, Q_m$ before it is utilized.
By replacing $Q_i$ with $3Q_i \oplus C_3$, we obtain a partition of $C_{3m+6}$.
Then three bumps and three dents of size $1 \times 1 \times 1$ are added only to the new $Q_1$.
Specifically, the coordinates $\{(-1,\, 1,\, 1),\, (1,\, -1,\, 1),\, (1,\, 1,\, -1)\}$ are added,
and the coordinates $\{(3m+5, \, 1, \, 1),\, (1, \, 3m+5,\, 1),\, (1, \, 1, \, 3m+5)\}$ are removed.
We define these new sets to be $Q_1',\, Q_2',\, \dots ,\, Q_m'$.
Figure 3 shows $Q_1',\, Q_2',\, Q_3'$ and $Q_4'$ for the case $m=3$, including $3 \times 3 \times 3$ cubes corresponding to free coordinates.

\begin{figure}[ht]
\centering
\includegraphics[width=0.95 \textwidth]{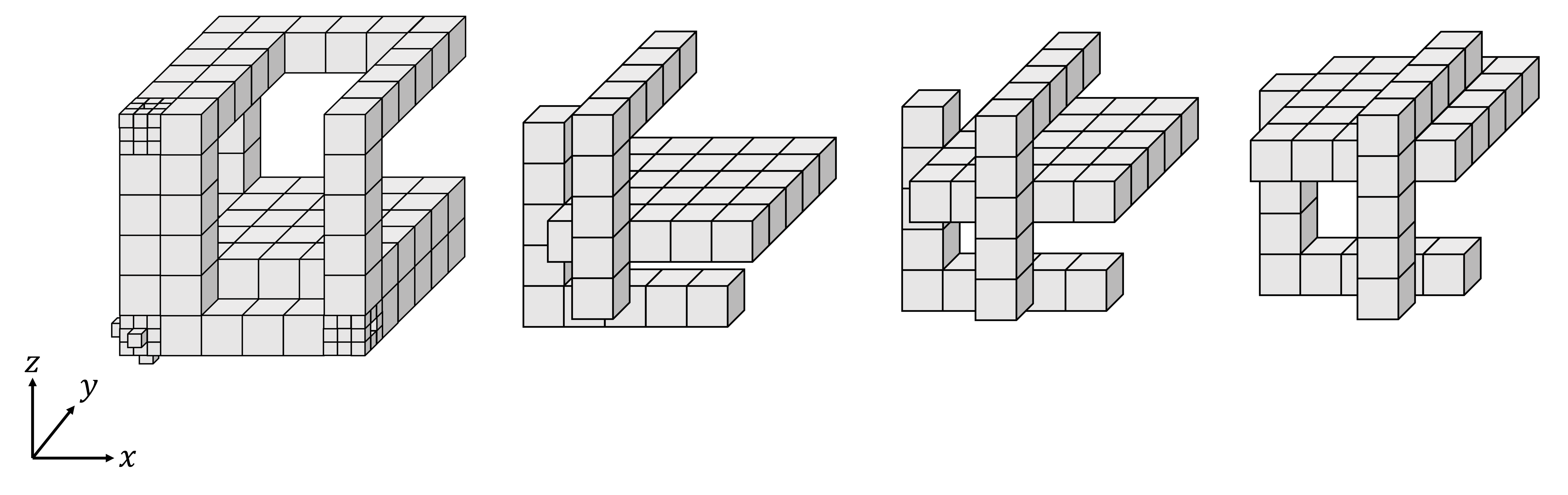}
\caption{$Q_1', \, Q_2', \, Q_3'$ and $Q_4'$(from left to right) for $m=4$}
\label{fig:Figure3}
\end{figure}

\subsection{Constructing the connected polycubes} 

For a given $l$, let $m = l^3 - (l-2)^3$ and $Q_1', \, Q_2', \, \dots , \, Q_m'$ be the corresponding sets from the construction in Section 2.1.
Let $x_1, \, x_2, \, \dots , \, x_m$ be all the elements of $C_l \setminus (C_{l-2} + (1, \, 1, \, 1))$.
We define $S_l$ as:
\[
S_l = \bigcup_{i=1}^{m} \left( Q_i' + (3m+6)x_i \right).
\]

$S_l$ can be thought as taking only the outermost unit cubes from $C_l$ and replacing them with a unique choice from $\{Q_1',\, Q_2',\, \dots ,\, Q_m'\}$.

\begin{lemma} 
$S_l$ is a connected polycube.
\end{lemma}

\begin{proof}
$Q_1',\, Q_2',\, \dots ,\, Q_m'$ are all connected polycubes and $\bigcup_{i=1}^{m} x_i$ is a connected polycube.
Thus, it is sufficient to show that if $x_i$ and $x_j$ are adjacent, then so are $Q_i' + (3m+6)x_i$ and $Q_j' + (3m+6)x_j$.
This is equivalent to $Q_i'$ being adjacent to $Q_j' + (3m+6)(x_j - x_i)$, which follows from the externally adjacent property.
\end{proof}

\begin{lemma} 
$S_l$ tiles $\mathbb{Z}^{3}$ only by $(3m+6)\mathbb{Z}^3$ up to translation.
\end{lemma}

\begin{proof}
If $S_l$ tiles $\mathbb{Z}^{3}$, the dents and bumps on $Q_1'$ force all copies of $Q_1'$ to lie on a lattice of side length $3m+6$.
Thus, $S_l$ must also lie on the same lattice. If copies of $S_l$ are translated by every vector in $(3m+6)\mathbb{Z}^3$,
no two tiles overlap since $Q_i' \cap Q_j' = \varnothing$ if $i \neq j$. Furthermore,
\[
\begin{aligned}
S_l \oplus (3m+6)\mathbb{Z}^3 &= \left( \bigcup_{i=1}^{m} (Q_i' + (3m+6)x_i) \right) \oplus (3m+6)\mathbb{Z}^3\\
&= \bigcup_{i=1}^{m} \left( Q_i' \oplus (3m+6)\mathbb{Z}^3 + (3m+6)x_i \right)\\
&= \bigcup_{i=1}^{m} \left(Q_i' \oplus (3m + 6)\mathbb{Z}^3 \right) \\
&= \left( \bigcup_{i=1}^{m} Q_i' \right) \oplus (3m+6)\mathbb{Z}^3 = \mathbb{Z}^3
\end{aligned}
\]
which shows that $S_l$ tiles $\mathbb{Z}^{3}$ by $(3m+6)\mathbb{Z}^3$.
\end{proof}

By Lemma 2.4, we can see that $S_l$ behaves similar to $C_{3m+6}$ which can only be translated by $(3m+6)\mathbb{Z}^3$.

\begin{theorem} 
Given a set of $k$ polycubes, some of which may be disconnected,
there is a set of $k$ connected polycubes such that the set of tilings by the two sets of polycubes have a one-to-one correspondence.
\end{theorem}

\begin{proof}
Let $P_1,\, P_2,\, \dots ,\, P_k$ be $k$ polycubes.
We choose an integer $l$ such that the given $k$ polycubes can each fit inside an $l \times l \times l$ cube.
We define new polycubes $P_1',\, P_2',\, \dots ,\, P_k'$ as
\[
P_i' = (3m+6)P_i \oplus S_l \quad (1 \leq i \leq k)
\]
where $m = l^3 - (l-2)^3$.
An example of this construction is shown in Figure 4; note that the $2$-dimensional counterpart of $S_l$ was used for easier visualization.

We know from Lemma 2.3 that $S_l$ is a connected polycube.
From the internally adjacent property, $S_l + (3m+6)u$ and $S_l + (3m+6)v$ are adjacent for $u,\, v \in C_l$ and $u \neq v$.
Since some translation of $P_i$ is a subset of $C_l$, every combination of two copies of $S_l$ in $P_i'$ are adjacent. Thus, $P_i'$ is a connected polycube.

Consider a tiling by $P_1',\, P_2',\, \dots ,\, P_k'$. By the definition of these polycubes, any tiling by $P_1',\, P_2',\, \dots ,\, P_k'$ can be seen as a tiling by $S_l$.
By Lemma 2.4, in order to form a tiling, copies of $S_l$ must be translated by every vector in $(3m+6)\mathbb{Z}^3$.
This corresponds to a tiling of $(3m+6)\mathbb{Z}^3$ by $(3m+6)P_1,\, (3m+6)P_2,\, \dots ,\, (3m+6)P_k$,
which directly corresponds to a tiling of $\mathbb{Z}^{3}$ by $P_1$, $P_2$,  $\dots$, $P_k$.
\end{proof}

\begin{figure}[ht]
\centering
\includegraphics[width=0.65 \textwidth]{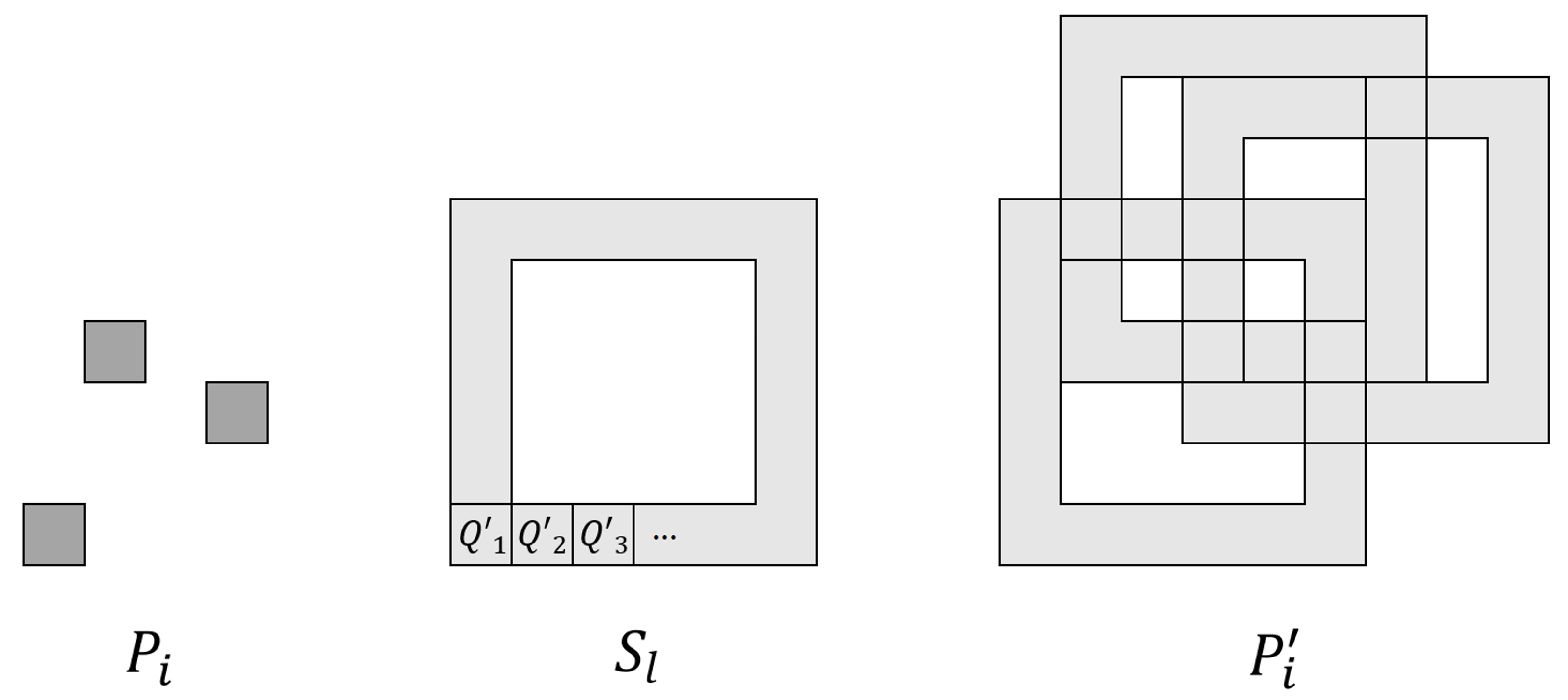}
\captionsetup{width=0.75 \textwidth}
\caption{An example of $P_i$, $S_l$ and $P_i'$ constructed using the $2$-dimensional counterpart of $S_l$}
\label{fig:Figure4}
\end{figure}

\section{Cyclic triomino problems and undecidability} 

In this section, we discuss the cyclic triomino problem, which is a variation of Greenfeld and Tao’s domino problem.
We explain the motivations behind creating this variation, then describe the problem in detail.
Lastly, we prove that the cyclic triomino problem is undecidable.

\subsection{The domino problem and the cyclic domino problem} 

First, we explain the domino problem proposed by Greenfeld and Tao \cite{RG}.
A \textbf{domino board} is the lattice $\mathbb{Z}^{2}$, and we define $e_1 = (1,\, 0)$ and $e_2 = (0,\, 1)$.
A \textbf{domino set} $\cR$ is a 3-tuple
\[
\left( \cW, \, \cR_1, \, \cR_2 \right)
\]
where $\cW$ is a non-empty finite set and $\cR_1,\, \cR_2 \subset \cW^2$.
The $\cR$-\textbf{domino problem} asks if there exists a function $\cT \colon \mathbb{Z}^{2} \to \cW$ such that
\[
\left( \cT(s),\, \cT(s + e_i) \right) \in \cR_i
\]
for every $s \in \mathbb{Z}^{2}$.
If $\cT$ is one such function, we say that the $\cR$-domino problem is solvable and $\cT$ is a solution(Figure 5).

The $\cR$-domino problem can be viewed as placing one element of $\cW$ on all points on the domino board,
with the condition that every pair of horizontally or vertically adjacent elements must belong in $\cR_1$ and $\cR_2$ respectively.
Greenfeld and Tao showed that this problem is undecidable from the undecidability of Wang’s domino problem \cite{RG}.

\begin{figure}[ht]
\centering
\includegraphics[width=0.8 \textwidth]{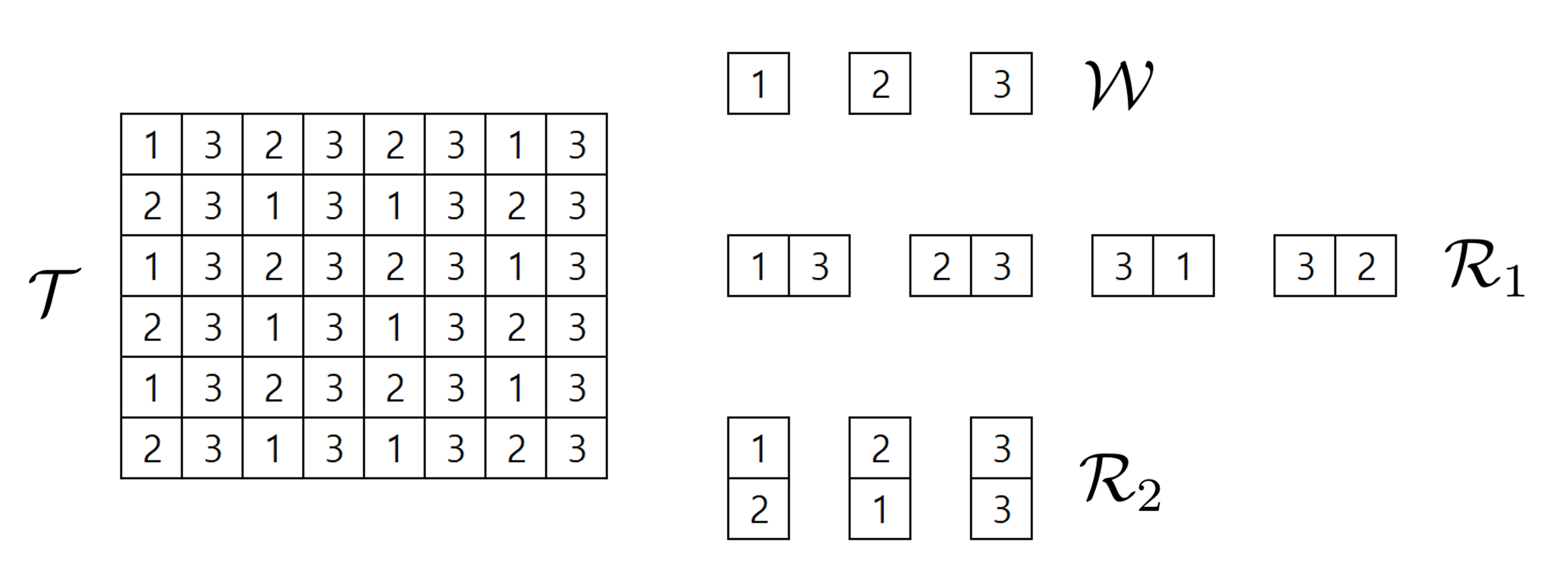}
\captionsetup{width=1 \textwidth}
\caption{An example of a domino problem and its solution}
\label{fig:Figure5}
\end{figure}

However, when we attempt to encode the domino problem with a small number of polycubes, a problem arises.
One may consider using a $1 \times 1 \times n$ cuboid tile, while encoding the choices from $\cW$ as the $z$-coordinates of each cuboid(Figure 6).
If we take one polycube from each row of polycubes repeating vertically, this approach provides a very similar environment to that of the domino problem.

\begin{figure}[ht]
\centering
\includegraphics[width=0.2 \textwidth]{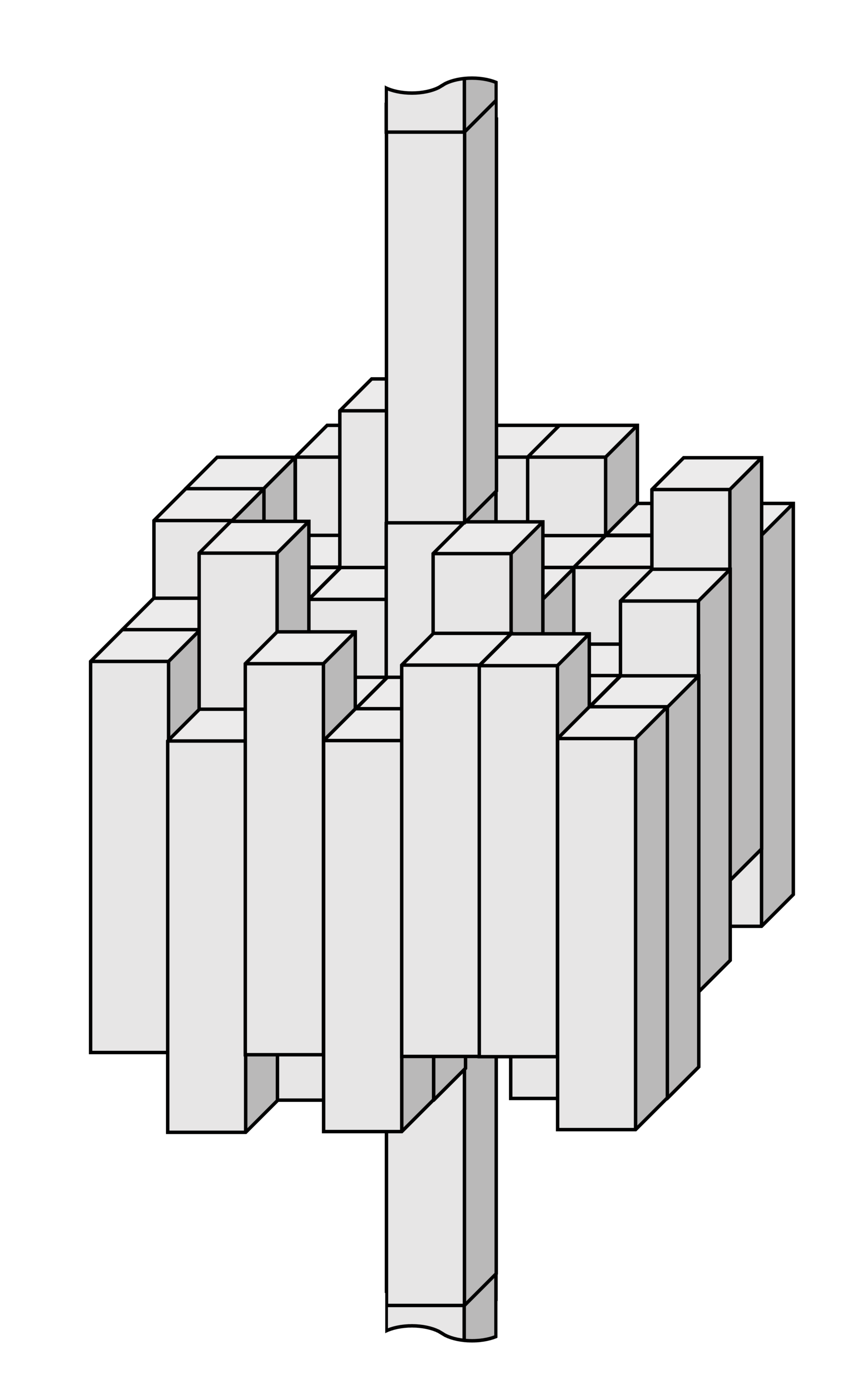}
\caption{A tiling by a $ 1 \times 1 \times n$ cuboid}
\label{fig:Figure6}
\end{figure}

Although it may seem that we can add bumps and dents to these cuboids to enforce the rules of the domino problem,
the set of rules that can be encoded by this approach is very limited.
Since every $1 \times 1 \times n$ cuboid must be identical, if two cuboids located at $z = 1$ and $z = 2$ can be horizontally adjacent,
then so can two cuboids at $z = 1 + k$ and $z = 2 + k$ for some integer $k$.
In other words, $(1,\, 2) \in \cR_1$ implies $(1 + k,\, 2 + k) \in \cR_1$.
The set of rules that can be encoded by this approach can be modeled by the cyclic domino problem.

The \textbf{cyclic domino set} is a domino set $\left( \cW, \, \cR_1,\, \cR_2 \right)$ where
\[
\cW = \mathbb{Z}_n
\]
and
\[
(a,\, b) \in \cR_i ~ \to ~ (a + k,\, b + k) \in \cR_i
\]
for every $k \in \cW$.
If $\cR$ is a cyclic domino set, we say that the $\cR$-\textbf{cyclic domino problem} is solvable if the $\cR$-domino problem is solvable(Figure 7).

\begin{figure}[ht]
\centering
\includegraphics[width=0.68 \textwidth]{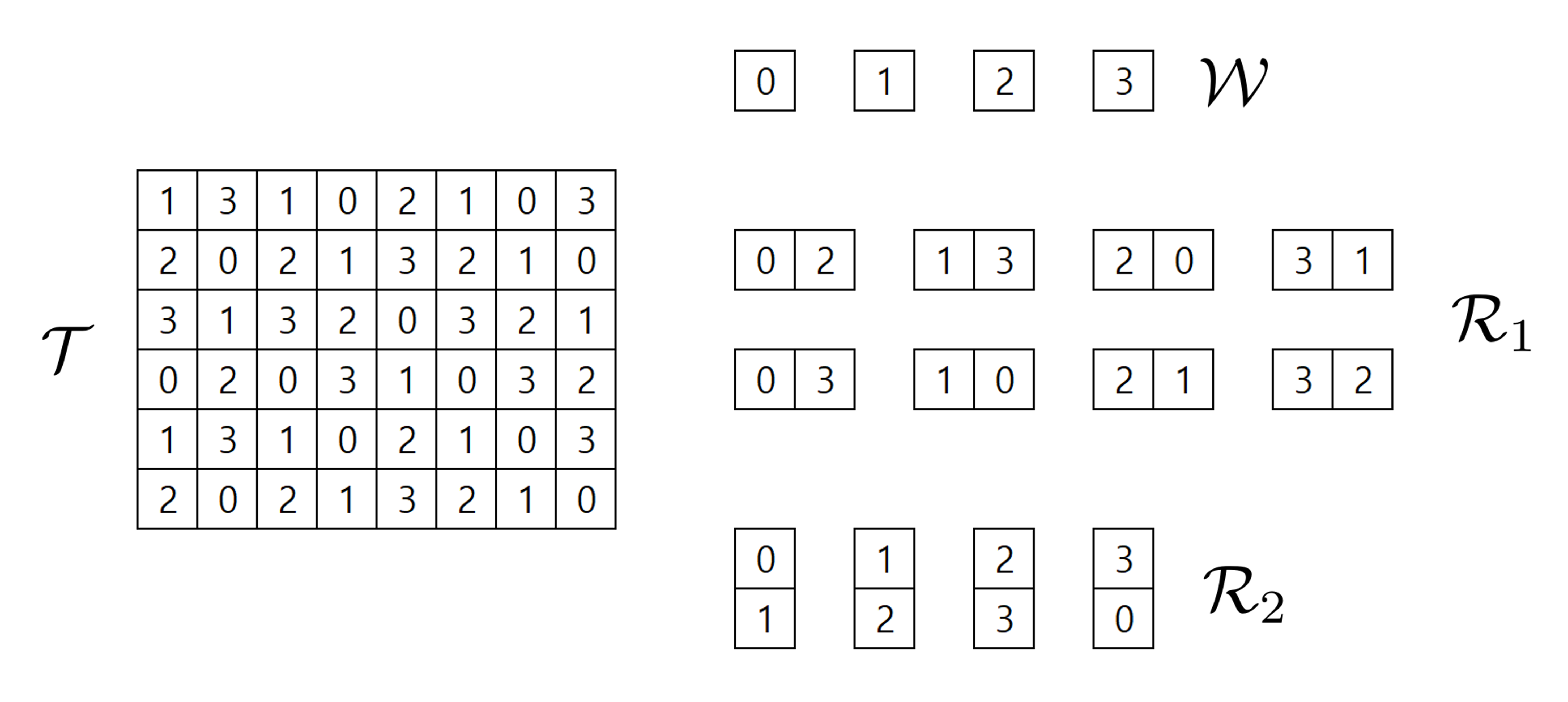}
\captionsetup{width=1 \textwidth}
\caption{An example of a cyclic domino problem and its solution}
\label{fig:Figure7}
\end{figure}

It can easily be proven that the $\cR$-cyclic domino problem is solvable as long as $\cR_1 \neq \varnothing$ and $\cR_2 \neq \varnothing$.
Thus, the cyclic domino problem is decidable.
Instead of replacing the encoding method with one that does not have a cyclic nature, we extend the set of restrictions from the cyclic domino problem.

\subsection{The cyclic triomino problem} 

A \textbf{triomino board} is the lattice $\mathbb{Z}^{2}$, and let $u_1 = (1,\, 0)$, $u_2 = (0,\, 1)$, $u_3 = (-1,\, 0)$, $u_4 = (0,\, -1)$.
A \textbf{cyclic triomino set} $\cS$ is a 5-tuple
\[
\left( \cV,\, \cS_1,\, \cS_2,\, \cS_3,\, \cS_4 \right)
\]
where
\[
\cV = \mathbb{Z}_n
\]
\[
\cS_1, \, \cS_2, \, \cS_3, \, \cS_4 \subset {\cV}^3
\]
and
\[
(a,\, b,\, c) \in \cS_i ~ \to ~(a + k,\, b + k,\, c + k) \in \cS_i
\]
for every $k \in \cV$.
For convenience, we define $u_{4q+r} = u_r$ and $\cS_{4q+r} = \cS_r$ for $q,\, r \in \mathbb{Z}$ and $1 \leq r \leq 4$.

The $\cS$-\textbf{cyclic triomino problem} asks if there exists a function $\cT \colon \mathbb{Z}^{2} \to \cV$ such that
\[
\left( \cT(s),\, \cT(s + u_i),\, \cT(s + u_{i + 1}) \right) \in \cS_i
\]
for every $s \in \mathbb{Z}^{2}$ and every $1 \leq i \leq 4$.
If $\cT$ is one such function, we say that $\cS$-cyclic triomino problem is solvable and $\cT$ is a solution(Figure 8).
The cyclic triomino problem can be viewed as replacing the domino-shaped restrictions from the cyclic domino problem to L-triomino shaped restrictions.

\begin{figure}[ht]
\centering
\includegraphics[width=0.75 \textwidth]{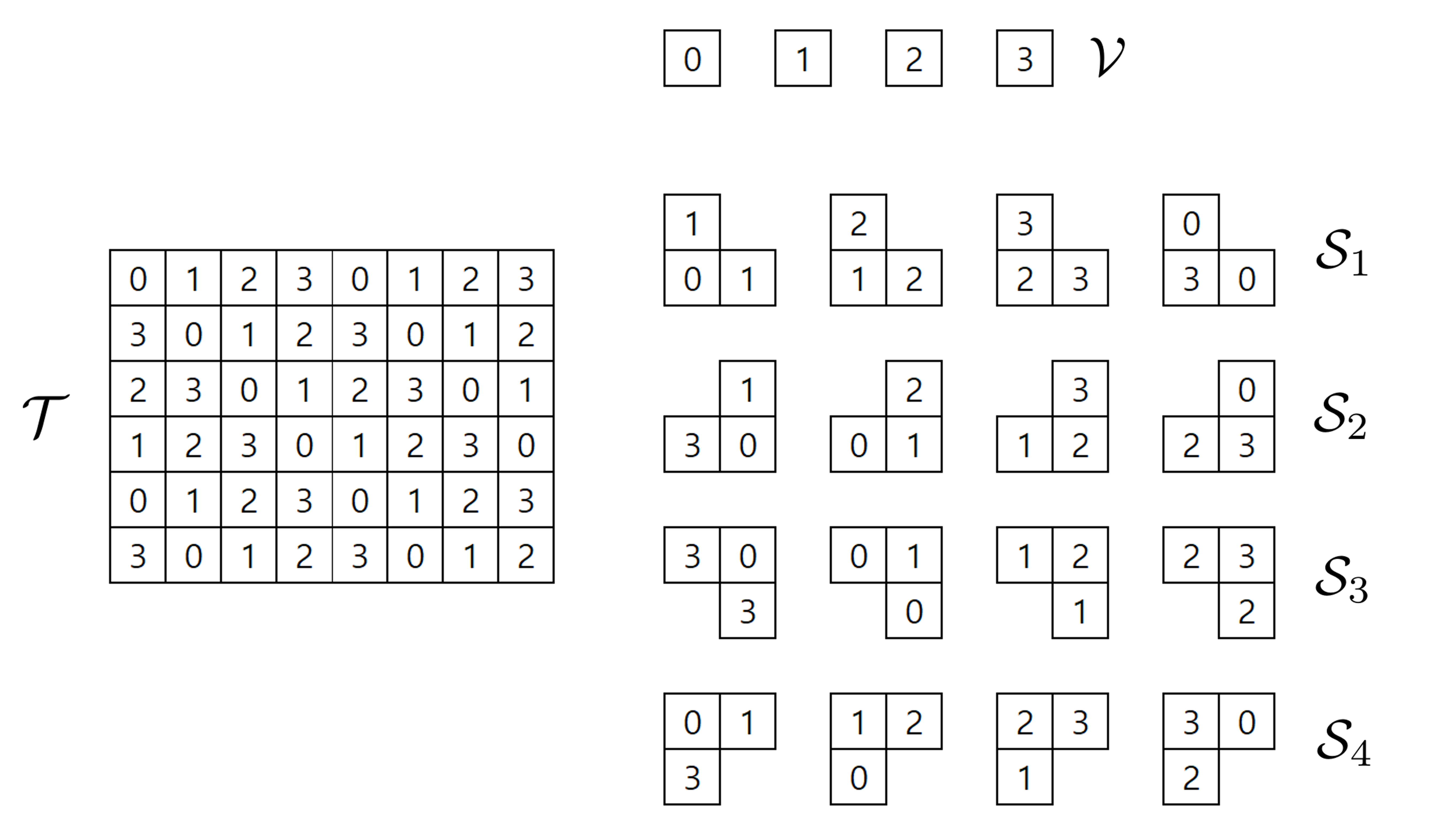}
\captionsetup{width=1 \textwidth}
\caption{An example of a cyclic triomino problem and its solution}
\label{fig:Figure8}
\end{figure}

\begin{theorem} 
The cyclic triomino problem is undecidable.
\end{theorem}

To prove Theorem 3.1, we construct a cyclic triomino set encoding an arbitrary domino problem.

Let $\cR = (\cW , \, \cR_1, \, \cR_2)$ be a given domino set.
We may assume $\cW = \{ 1,\, 2,\, \dots ,\, m \}$ for some $m \in \mathbb{N}$. Let
\[
L = \{ (w,\, 0,\, 0) \mid w \in \cW \},
\]
\[
K_1 = L \cup \{ (0,\, b,\, a) \mid (a,\, b) \in \cR_1 \},
\]
\[
K_2 = L \cup \{ (0,\, a,\, b) \mid (a,\, b) \in \cR_2 \},
\]
\[
K_3 = L \cup \{ (0,\, a,\, b) \mid (a,\, b) \in \cR_1 \},
\]
\[
K_4 = L \cup \{ (0,\, b,\, a) \mid (a,\, b) \in \cR_2 \}.
\]

Here we have defined the ‘representatives’ $K_i$ of each cyclic restriction.
Although the restrictions themselves are cyclic, only $K_i$’s will be able to appear in the triomino board(up to shifting all numbers on the triomino board by an integer).
Now we define the cyclic triomino set $\cS = \left( \cV,\, \cS_1,\, \cS_2,\, \cS_3,\, \cS_4 \right)$ where
\[
\cV = \mathbb{Z}_n \quad (n \geq 2m + 1),
\]
\[
\cS_i = \{ (a + k,\, b + k,\, c + k) \mid (a,\, b,\, c) \in K_i,\, k \in \cV \}.
\]
Figure 9 shows an example of $\cS$ constructed from the domino set in Figure 5.

\begin{figure}[ht]
\centering
\includegraphics[width=1 \textwidth]{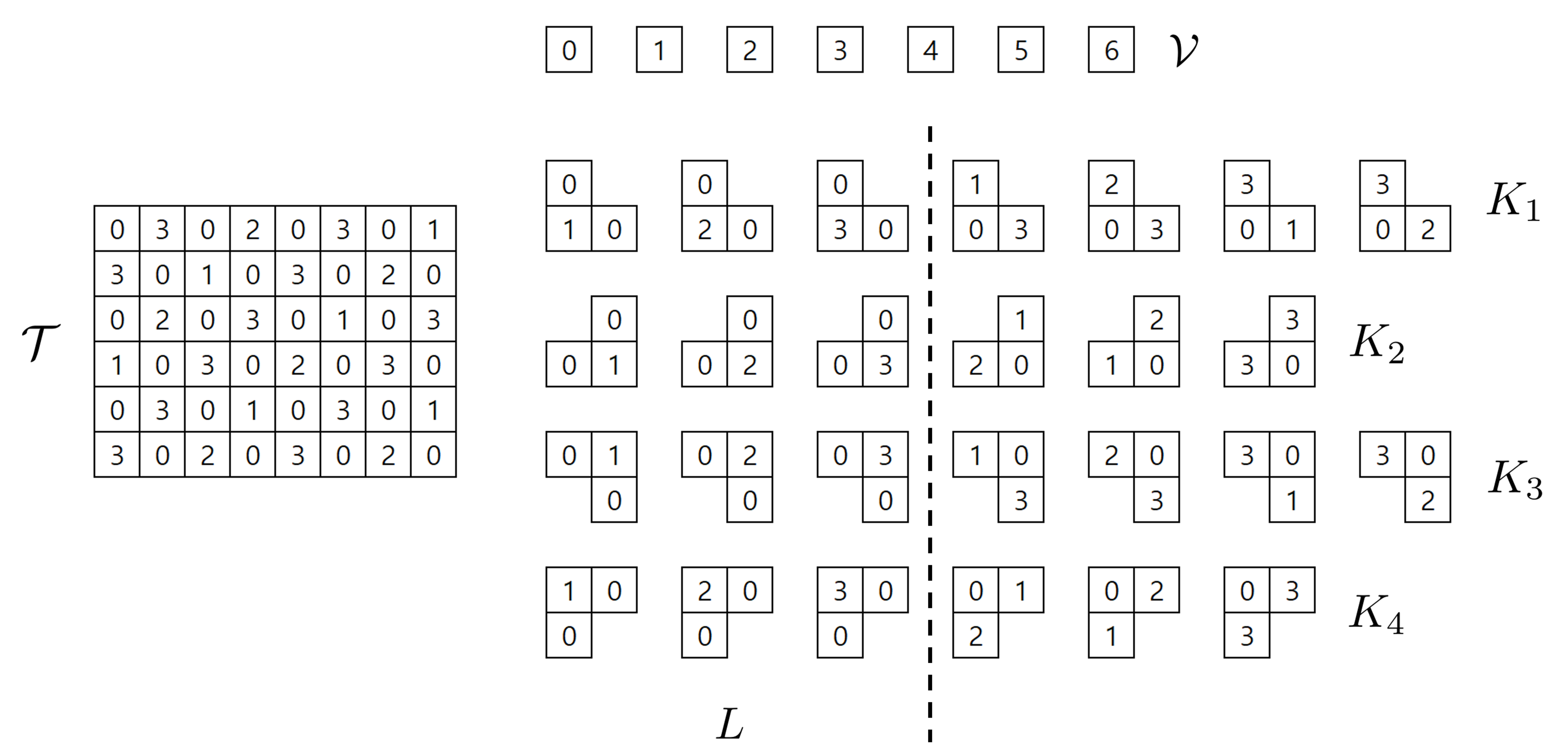}
\captionsetup{width=0.95 \textwidth}
\caption{A cyclic triomino set encoding the domino problem in Figure 5; only $K_i$ is shown from ${\cS}_i$.}
\label{fig:Figure9}
\end{figure}

Let $\cT$ be a solution to $\cS$-cyclic triomino problem.
For $s \in \mathbb{Z}^{2}$ and $i \in \mathbb{Z}$, define $p(s,\, i)$ and $q(s,\, i)$ as the statements
\[
\left( \cT(s),\, \cT(s + u_i),\, \cT(s + u_{i + 1}) \right) \in L,
\]
\[
\left( \cT(s),\, \cT(s + u_i),\, \cT(s + u_{i + 1}) \right) \in K_i \setminus L
\]
respectively.

$p(s,\, i)$ and $q(s,\, i)$ can be thought of as examining the L-triomino with position $s$ and orientation $i$.
Note that neither $p(s,\, i)$ nor $q(s,\, i)$ might be true, as the condition of $\cS$-cyclic triomino problem is that
\[
\left( \cT(s),\, \cT(s + u_i),\, \cT(s + u_{i + 1}) \right) \in {\cS_i}.
\]

\begin{lemma} 
Let $s \in \mathbb{Z}^2$ be fixed. If $p(s,\,i)$ is true for some $1 \leq i \leq 4$, then it is true for every $1 \leq i \leq 4$.
\end{lemma}

\begin{proof}
Since $p(s,\,i)$ and $p(s,\,i + 4)$ are identical statements, it is sufficient to show that $p(s,\,i) \to p(s,\,i + 1)$ for every $1 \leq i \leq 4$.

Suppose that $p(s,\,i)$ is true. By examining $L$, we see that
\[
1 \leq \cT(s) \leq m,\, \cT(s + u_{i + 1}) = 0. \qquad (\mathord{*})
\]

We also know that the following must hold:
\[
(\cT(s),\, \cT(s + u_{i + 1}),\, \cT(s + u_{i + 2})) \in {\cS_{i + 1}}
\]

Let $(a,\, b,\, c)$ be the representative of $(\cT(s),\, \cT(s + u_{i + 1}),\, \cT(s + u_{i + 2}))$ in $K_{i + 1}$. Then
\[
a - b = \cT(s) - \cT(s + u_{i + 1})
\]
holds (in $\cV$).
Combining this with ($\mathord{*}$), we get:
\[
1 \leq a - b \leq m
\]

Since we chose $\cV=\mathbb{Z}_n$ such that $n \geq 2m + 1$,  
$(a,\, b,\, c)$ fails to meet this condition if  $(a,\, b,\, c) \in K_{i + 1} \setminus L$.
Thus $(a,\, b,\, c) \in L$, and $\bigl(\cT(s),\, \cT(s + u_{i + 1}),\, \cT(s + u_{i + 2})\bigr)$ is equal to an element in $L$ shifted by some integer.  
Comparing the second element of the tuple, we see that every element in $L$ has a second element of $0$,  
which is also true for $\bigl(\cT(s),\, \cT(s + u_{i + 1}),\, \cT(s + u_{i + 2})\bigr)$.
Therefore, $\bigl(\cT(s),\, \cT(s + u_{i + 1}),\, \cT(s + u_{i + 2})\bigr) \in L$.
\end{proof}

\begin{lemma} 
Let $s \in \mathbb{Z}^2$ be fixed. If $q(s,\,i)$ is true for some $1 \leq i \leq 4$, then it is true for every $1 \leq i \leq 4$.
\end{lemma}

\begin{proof}
We repeat the argument in lemma 3.2 by noticing that
\[
1 \leq \cT(s + u_{i + 1}) \leq m,\, \cT(s) = 0
\]
and comparing the first element of the tuple at the end of the proof.
\end{proof}

\begin{lemma} 
If $p(s,\,1)$ is true, $q(s + u_j,\,1)$ is true for every $1 \leq j \leq 4$. Likewise, if $q(s,\,1)$ is true, $p(s + u_j,\,1)$ is true for every $1 \leq j \leq 4$.
\end{lemma}

\begin{proof}
Suppose that $p(s,\,1)$ is true. By Lemma 3.2, $p(s,\,j)$ is true for every $j$.
We can repeat the argument in Lemma 3.2 by noticing that
\[
1 \leq \cT(s) \leq m,\, \cT(s + u_j) = 0,
\]
\[
(\cT(s + u_j),\, \cT(s + u_j + u_{j + 1}),\, \cT(s)) \in {\cS_{j + 1}}
\]
and comparing the first element of the tuple at the end of the proof.
This shows that $q(s + u_j, \, j + 1)$ is true, from which we obtain that $q(s + u_j, \,1)$ is true by Lemma 3.3.

Now suppose that $q(s,\,1)$ is true. By Lemma 3.3, $q(s,\,j)$ is true for every $j$.
We can repeat the argument in Lemma 3.2 by noticing that
\[
1 \leq \cT(s + u_j) \leq m,\, \cT(s) = 0,
\]
\[
(\cT(s + u_j),\, \cT(s + u_j + u_{j + 1}),\, \cT(s)) \in {\cS_{j + 1}}
\]
and comparing the third element of the tuple at the end of the proof.
This shows that $p(s + u_j,\, j + 1)$ is true, from which we obtain that $p(s + u_j,\,1)$ is true by Lemma 3.2.
\end{proof}

By combining the properties from the above 3 lemmas, we obtain the general structure of $\cT$.
This allows us to prove Lemma 3.5.

\begin{lemma} 
$\cS$-cyclic triomino problem is solvable if and only if $\cR$-domino problem is solvable.
\end{lemma}

\begin{proof}
Suppose that $\cS$-cyclic triomino problem is solvable with solution $\cT$.
Consider the following 3-tuple:
\[
t = (\cT((0,\,0)),\, \cT((1,\,0)),\, \cT((0,\,1)))
\]

We know that $t \in {\cS_1}$.
Since $\cT^{\prime}(s) = \cT(s) + k$ gives a solution for any $k \in \cV$, we can choose $\cT$ such that $t \in K_1$.

Now we see that either $p((0,\,0),\,1)$ or $q((0,\,0),\,1)$ is true.
By using Lemmas 3.2 -- 3.4, the following is obtained for $x,\, y,\, i \in \mathbb{Z}$:
\[
\begin{aligned}
x + y \equiv \delta &\pmod{2} \to p((x,\,y),\,i) \\
x + y \equiv 1 - \delta &\pmod{2} \to q((x,\,y),\,i)
\end{aligned}
\]
where $\delta \in \{0,\,1\}$ is a constant. From this we get:
\[
\begin{aligned}
x + y \equiv \delta &\pmod 2 \rightarrow 1 \le \cT((x, y)) \le m \\
x + y \equiv 1 - \delta &\pmod 2 \rightarrow \cT((x, y)) = 0
\end{aligned}
\]

Thus, the non-zero values on the triomino board form a lattice of side length $\sqrt{2}$.
By observing $K_i \setminus L$, we can see that $K_i \setminus L$ enforces the restrictions of $\cR$ on this lattice.
By taking the values on this lattice, we obtain a solution to $\cR$-domino\, problem.

If $\cR$-domino problem is solvable, it is easy to verify that a solution to $\cS$-cyclic triomino problem exists in the form of the above configuration.
Therefore, $\cS$-cyclic triomino problem is solvable if and only if $\cR$-domino problem is solvable.
\end{proof}

From Lemma 3.5, it follows that the cyclic triomino problem is undecidable.
This proves Theorem 3.1.

\section{Encoding the cyclic triomino problem with polycubes} 

Now that we have the results of Theorem 2.5 and Theorem 3.1,
we are ready to prove the undecidability of tiling $\mathbb{Z}^3$ with a set of two connected polycubes.

\begin{theorem}
It is undecidable whether a set of two connected polycubes can tile $\mathbb{Z}^3$.
\end{theorem}

In order to prove Theorem 4.1, we will construct two disconnected polycubes encoding an arbitrary cyclic triomino set.
These can then be turned into connected polycubes using Theorem 2.5.

\subsection{The blockers and the towers} 

The \textbf{blockers} are 1-dimensional tiles which provide a way to enforce the 3-way matching rules of the cyclic triomino problem.
There are three types of blockers, each with an ‘on’ and an ‘off’ state(Table~\ref{tab:blockers}):

\begin{table}[h]
\centering
\begin{tabular}{|c|c|c|}
\hline
 & off & on \\
\hline
type 1 & $\alpha_0 = \varnothing$ & $\alpha_1 = \{0,\,5\}$ \\
\hline
type 2 & $\beta_0 = \{1,\,4\}$ & $\beta_1 = \varnothing$ \\
\hline
type 3 & $\gamma_0 = \varnothing$ & $\gamma_1 = \{2,\,3\}$ \\
\hline
\end{tabular}
\caption{Notation of blockers}
\label{tab:blockers}
\end{table}

For $i, \, j, \, k \in \{0, \, 1\}$, consider tiling the set
\[
S = \mathbb{Z} \setminus 6\mathbb{Z} \oplus (\alpha_i \cup \beta_j \cup \gamma_k)
\]
with the tile
\[
D = \{0,\,1\}.
\]

This is equivalent to tiling $\mathbb{Z}$ while $\alpha_i,\,\beta_j$ and $\gamma_k$ have already been placed with a period of $6$.
By checking every combination of $(i,\,j,\,k)$, we see that $D$ can tile $S$ except for the case $(i,\,j,\,k) = (1,\,1,\,1)$(Figure 10).
We call this process of filling in the leftover gaps from the blockers a \textbf{verification}.

\begin{figure}[ht]
\centering
\includegraphics[width=1 \textwidth]{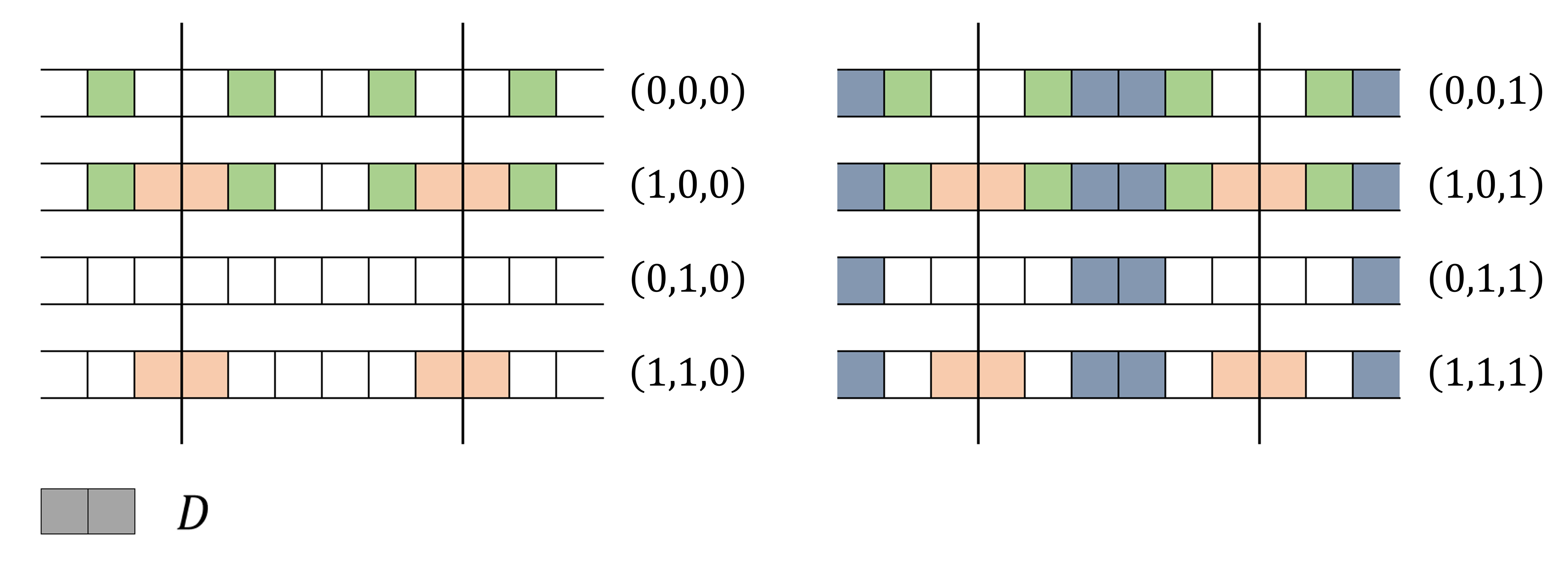}
\captionsetup{width=1 \textwidth}
\caption{$S$ for every choice of $(i, \, j, \, k)$; $S$ corresponds to empty squares.}
\label{fig:Figure10}
\end{figure}

Let $n$ be a positive integer that satisfies $\gcd(n,\,6)=1$.
The tiles constructed from this point will depend on the choice of $n$.
These tiles will later be used to encode the cyclic triomino set $(\cV,\,\cS_1,\,\cS_2,\,\cS_3,\,\cS_4)$ where $\cV=\mathbb{Z}_n$.

The \textbf{towers} are 1-dimensional tiles composed of blockers.
While single blockers forbid a specific triple of on/offs, the towers forbid specific triples of positions ranging from $0$ to $n-1$.

We define the tower $\alpha_T$ as
\[
\alpha_T = \bigcup_{i=0}^{n-1} \left( n \alpha_{\delta(i)} + 6i \right)
\]
where $\delta(i)=1$ if $i=0$ and $0$ otherwise. Similarly:
\[
\beta_T = \bigcup_{i=0}^{n-1} \left( n \beta_{\delta(i)} + 6i \right), \quad
\gamma_T = \bigcup_{i=0}^{n-1} \left( n \gamma_{\delta(i)} + 6i \right)
\]

Intuitively, $\alpha_T$ could be seen as a layout of $\alpha_1 \alpha_0 \alpha_0 \cdots \alpha_0$ but with each blocker interlacing.
No two blockers in $\alpha_{T}$ can overlap because we chose $n$ so that $\gcd(n,\,6)=1$.

Since the blockers have been scaled by $n$, we assume that the towers are placed with a period of $6n$.
We also assume that the towers can only be translated by an element of $6\mathbb{Z}$.
For $0\leq a,\,b,\,c \leq n-1$, consider tiling the set
\[
S' = \mathbb{Z} \setminus 6n\mathbb{Z} \oplus \left( (\alpha_T + 6a) \cup (\beta_T + 6b) \cup (\gamma_T + 6c) \right)
\]
with the tile
\[
nD = \{0, \, n\}
\]
(Figure 11). Notice that
\[
\begin{aligned}
6n\mathbb{Z} \oplus \left( \alpha_{T} + 6a \right)
&= 6n\mathbb{Z} \oplus \left( \bigcup_{i = a}^{n - 1 + a} \left( n\alpha_{\delta(i - a)} + 6i \right) \right) \\
&= \bigcup_{i = a}^{n - 1 + a} \left( \left( 6n\mathbb{Z} + 6i \right) \oplus n\alpha_{\delta(i - a)} \right)
= \bigcup_{i = 0}^{n - 1} \left( \left( 6n\mathbb{Z} + 6i \right) \oplus n\alpha_{\delta(i - a)} \right)
\end{aligned}
\]
which means
\[
\begin{aligned}
S' &= \mathbb{Z} \setminus \bigcup_{i = 0}^{n - 1} 
\left( \left( 6n\mathbb{Z} + 6i \right) \oplus n\left( \alpha_{\delta(i - a)} \cup \beta_{\delta(i - b)} \cup \gamma_{\delta(i - c)} \right) \right) \\
&= \bigcup_{i = 0}^{n - 1} \left( n \left( \mathbb{Z} \setminus 6\mathbb{Z} \oplus
\left( \alpha_{\delta(i - a)} \cup \beta_{\delta(i - b)} \cup \gamma_{\delta(i - c)} \right) \right) + 6i \right).
\end{aligned}
\]
Thus, attempting to tile $S'$ with $nD$ results in $n$ independent verifications in $n$ copies of $n\mathbb{Z}$.  
From the property of blockers, the only case where $nD$ cannot tile $S'$ is when
\[
i - a = i - b = i - c = 0
\]
for some $0 \leq i \leq n - 1$; in other words,
\[
a = b = c.
\]

\begin{figure}[ht]
\centering
\includegraphics[width=1 \textwidth]{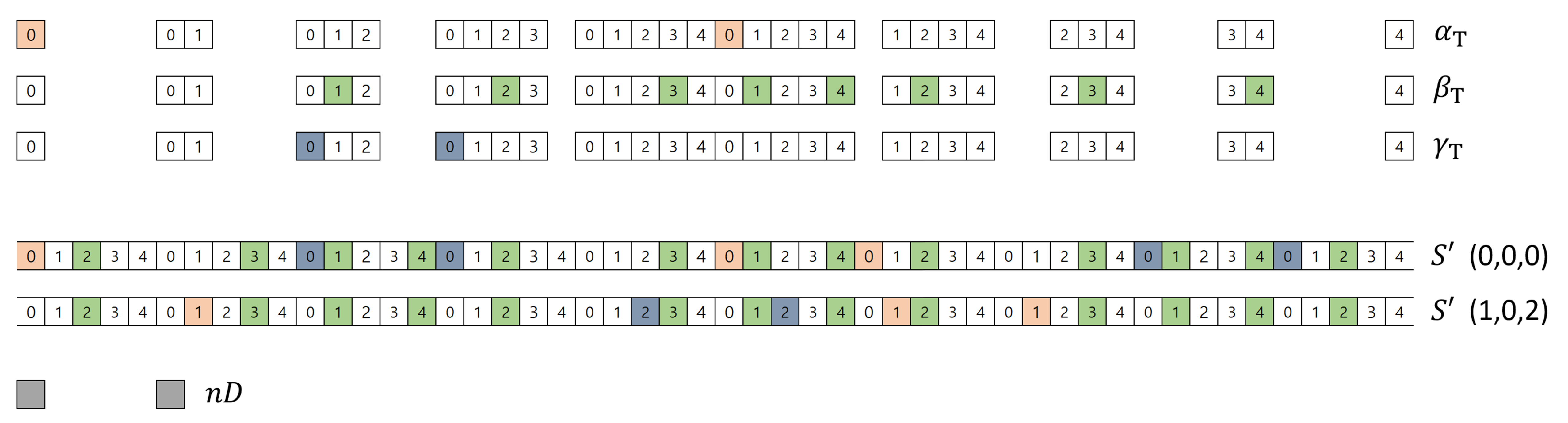}
\captionsetup{width=0.9 \textwidth}
\caption{$S'$ for cases $(a,\,b,\,c)=(0,\,0,\,0)$ and $(1,\,0,\,2)$ when $n=5$;
same numbers correspond to the same copy of $n\mathbb{Z}$.}
\label{fig:Figure11}
\end{figure}

For the general case where $a,\, b,\, c \in \mathbb{Z}$,\, $nD$ cannot tile $S'$ if and only if
\[
a \equiv b \equiv c \pmod{n}.
\]

This is because $S'$ is identical under $a$, $b$, and $c$ being altered by a multiple of $n$.
With this property of towers, we have the ability to forbid $n$ triples of positions in a cyclic manner.
Using Figure 6 as a base structure, we will insert towers into the cuboids to forbid every triple which is not allowed in the given cyclic triomino set.

\subsection{The filler and the brick} 

First, we define a 3-dimensional version of the towers.
In the process, we modify the towers so that the polycube used for filling in gaps cannot tile $\mathbb{Z}^3$ by itself.
Let $O$ be the following polyomino:
\[
\{-1,\,0,\,1\}^2 \setminus \{(0,\,0)\}
\]

The \textbf{filler} is defined as
\[
O \times \{0,\,n\},
\]
and it tiles the leftover space from
\[
O \times \alpha_T, \quad O \times \beta_T, \quad O \times \gamma_T.
\]

\begin{lemma} 
The filler cannot tile $E \subset \mathbb{Z}^3$ if $E$ contains any translation of $\{-1,\,0,\,1\}^2 \times \{0\}$ as a subset.
\end{lemma}

\begin{proof}
Suppose that the filler can tile $E$ and let $(\{-1,\,0,\,1\}^2 \times \{0\}+v) \subset E$.
If a filler covers the coordinate $v$, the hole of this filler is positioned inside $\{-1,\,0,\,1\}^2 \times \{0\}+v$.
This hole cannot be covered, leading to a contradiction.
\end{proof}

The filler is one of the two polycubes used to encode the cyclic triomino problem(We actually scale up the filler by a factor of $3$,
but this is irrelevant until the end).
The other one is the brick, which we will explain now.

Let $\cS=(\cV,\,\cS_1,\,\cS_2,\,\cS_3,\,\cS_4)$ be a cyclic triomino set which satisfies $\cV=\mathbb{Z}_n$ and $\gcd(n,\,6)=1$.
This extra requirement does not affect the undecidability proof in Section 3, since the proof works as long as $n$ is unbounded.

For $1 \leq i \leq 4$, let $\overline{K_i}$ be the representatives of $\mathbb{Z}_n^3 \setminus \cS_i$ such that
\[
\overline{K_i} \oplus \{(k,\,k,\,k)\mid k\in \mathbb{Z}_n\}
\]
is non-overlapping and equals $\mathbb{Z}_n^3 \setminus \cS_i$.\, Then we need
\[
m = \left|\overline{K_1}\right| + \left|\overline{K_2}\right| + \left|\overline{K_3}\right| + \left|\overline{K_4}\right|
\]
towers to forbid every triple not in $\cS_i$.

The \textbf{empty brick} is a $(3m+2)\times 5 \times 6n$ cuboid with the inside carved.
It is defined as
\[
B_0 = I_{0,\,3m+1} \times I_{0,\,4} \times I_{0,\,6n-1} \setminus
\bigcup_{i=0}^{m-1} \left( O \times I_{0,\,6n-1} + (3i+2,\,2,\,0) \right)
\]
where $I_{a,\,b} = \{x\mid x\in \mathbb{Z},\, a\leq x\leq b\}$.
The empty brick has $m$ ‘poles’, which will fit $m$ towers(Figure 12).
To make the discussion more concise, we assume that any tiling by the empty brick or its variation has a period of $6n$ in the $z$ direction, and can only be translated by
\[
\{ ((3m+2)x, \, 5y, \, 6z) \mid x,\,y,\,z\in \mathbb{Z} \}.
\]
Considering the side lengths of the empty brick,
these two conditions make the tiling structure of the empty brick equivalent to the tiling structure of a $1\times 1\times n$ cuboid(Figure 6).
These conditions will be enforced later.

\begin{figure}[ht]
\centering
\includegraphics[width=0.22 \textwidth]{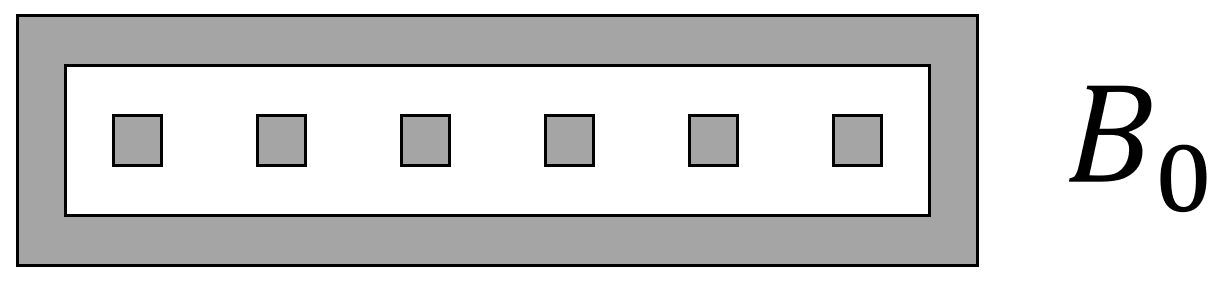}
\captionsetup{width=1 \textwidth}
\caption{A cross section of the empty brick along the $xy$ plane when $m=6$; the cross section is the same for every $0 \leq z \leq 6n-1$.}
\label{fig:Figure12}
\end{figure}

For a given triple $(a,\,b,\,c)\in \overline{K_1}$, consider the empty brick with three towers added:
\[
\begin{aligned}
B_1 = B_0 &\cup \left( O \times \alpha_T + (2,\,2,\,-6a) \right) \\
&\cup \left( O \times \beta_T + (-3m,\,2,\,-6b) \right) \\
&\cup \left( O \times \gamma_T + (2,\,-3,\,-6c) \right)
\end{aligned}
\]

Note that two towers are outside of the empty brick(Figure 13).
These towers allow the polycube $B_1$ to interact with adjacent $B_1$'s on the $+x$ and $+y$ directions.
Specifically, when three $B_1$'s are placed at $(0,\,0,\,6i)$, $(3m+2,\,0,\,6j)$ and $(0,\,5,\,6k)$,
towers $\alpha_T$, $\beta_T$ and $\gamma_T$ meet with a displacement of $6i-6a$, $6j-6b$ and $6k-6c$ respectively.
From the result in Section 4.1, this implies that the filler cannot tile the leftover space if and only if
\[
i-a \equiv j-b \equiv k-c \pmod{n}.
\]
Thus, the above placements of towers forbid three bricks from taking positions $(a,\, b,\, c)$ from
$\overline{K_1}$(and $n$ triples of positions from $\mathbb{Z}_n^{3}\setminus \cS_1$).

\begin{figure}[ht]
\centering
\includegraphics[width=0.8 \textwidth]{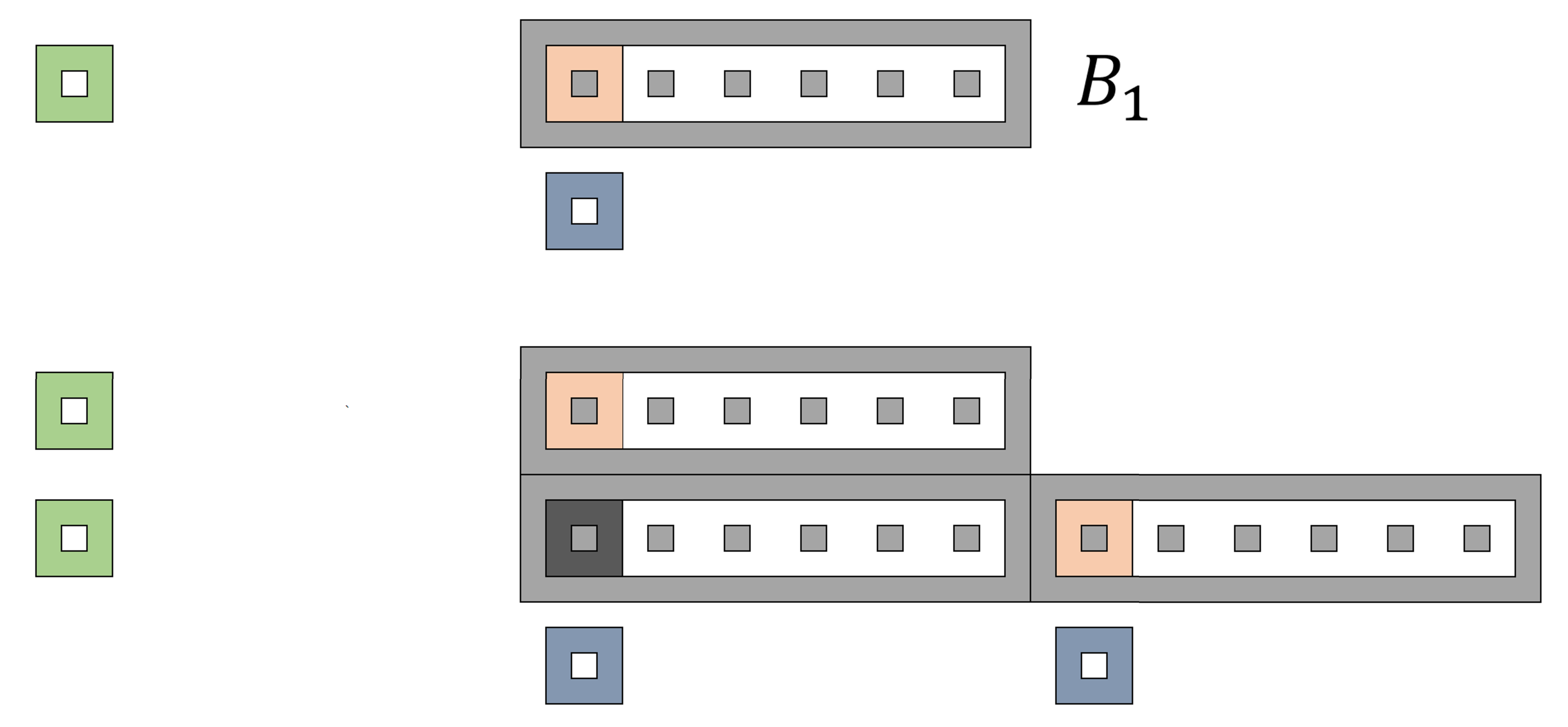}
\captionsetup{width=0.55 \textwidth}
\caption{A projection of $B_1$  on plane $xy$, and the interaction between three $B_1$’s}
\label{fig:Figure13}
\end{figure}

Similarly, we can forbid every triple of positions from $\overline{K_i}$ with $m$ sets of three towers.
Define $T_{0,\, i,\, j}$ to $T_{5,\, i,\, j}$ as
\[
\begin{aligned}
T_{0,\, i,\, j} &= O \times \alpha_T + (3i-1,\, 2,\, -6j),\\
T_{1,\, i,\, j} = T_{5,\, i,\, j} &= O \times \beta_T + (3i-3m-3,\, 2,\, -6j),\\
T_{2,\, i,\, j} &= O \times \gamma_T + (3i-1,\, -3,\, -6j),\\
T_{3,\, i,\, j} &= O \times \beta_T + (3i+3m+1,\, 2,\, -6j),\\
T_{4,\, i,\, j} &= O \times \gamma_T + (3i-1,\, 7,\, -6j).
\end{aligned}
\]

Let $(a_1,\, b_1,\, c_1),\, \dots,\, (a_m,\, b_m,\, c_m)$ be all the elements of
$\overline{K_1}$, $\overline{K_2}$, $\overline{K_3}$, and $\overline{K_4}$ in order and let $(a_i,\, b_i,\, c_i)$ be from $\overline{K_{g(i)}}$.
The brick is defined as
\[
B = B_0 \cup \bigcup_{i=1}^{m}
\left(
T_{0,\, i,\, a_i} \cup T_{g(i),\, i,\, b_i} \cup T_{g(i)+1,\, i,\, c_i}
\right)
\]
(Figure 14).

\begin{figure}[ht]
\centering
\includegraphics[width=0.8 \textwidth]{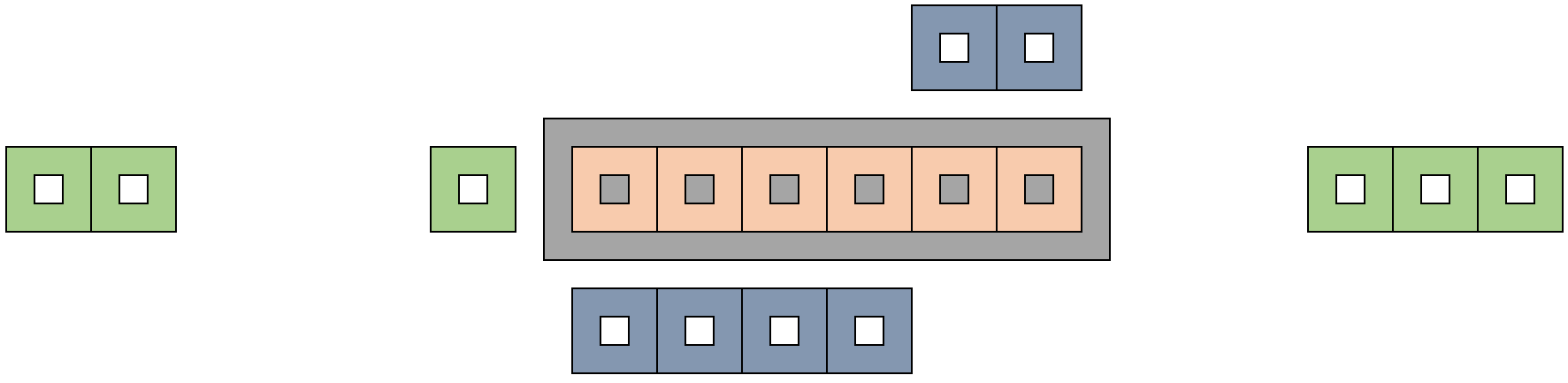}
\captionsetup{width=0.85 \textwidth}
\caption{A projection of $B$ on the $xy$ plane;  $\left|\overline{K_1}\right| = \left|\overline{K_2}\right| =2$ and $\left|\overline{K_3}\right| = \left|\overline{K_4}\right| =1$ in the figure.}
\label{fig:Figure14}
\end{figure}

\begin{lemma} 
If the filler tiles $\mathbb{Z}^3 \setminus (W \oplus B)$ where
\[
W \subset \left\{ \big( (3m+2)x,\, 5y,\, 6z \big) \mid x,\, y,\, z \in \mathbb{Z} \right\},
\]
\[
W + (0,\, 0,\, 6n) = W
\]
and $W \oplus B$ is non-overlapping, then $S$-cyclic triomino problem is solvable.
\end{lemma}

\begin{proof}
From Lemma 4.2, we see that there must be no gaps between bricks. Thus
\[
W = \left\{ \big( (3m+2)x,\, 5y,\, 6nz + 6 \cT(x,\, y) \big) \mid x,\, y,\, z \in \mathbb{Z} \right\}
\]
for a function $\cT \colon \mathbb{Z}^2 \to \{0,\, 1,\, \dots,\, n-1\}$.  
Let $u_1 = (1,\, 0)$, $u_2 = (0,\, 1)$, $u_3 = (-1,\, 0)$, $u_4 = (0,\, -1)$ and $s \in \mathbb{Z}^2$.  
From the argument for $B_1$,
\[
\big( \cT(s),\, \cT(s+u_1),\, \cT(s+u_2) \big) \notin \mathbb{Z}_n^3 \setminus \cS_1
\]
and thus belongs to $\cS_1$. Similarly,
\[
\big( \cT(s),\, \cT(s+u_2),\, \cT(s+u_3) \big) \in \cS_2,
\]
\[
\big( \cT(s),\, \cT(s+u_3),\, \cT(s+u_4) \big) \in \cS_3,
\]
\[
\big( \cT(s),\, \cT(s+u_4),\, \cT(s+u_1) \big) \in \cS_4.
\]
Therefore, $\cT$ is a solution to $\cS$-cyclic triomino problem.
\end{proof}

Finally, we enforce the assumptions made throughout the section.
We scale up the filler and the brick by a factor of $3$, then add bumps and dents to the brick.
For $0 \le i \le n-1$, the coordinates
\[
(-1,\, 1,\, 1+18i),\, (1,\, -1,\, 1+18i),\, (1,\, 1,\, -1)
\]
are added, and the coordinates
\[
(9m+5, \, 1,\, 1+18i),\, (1,\, 14,\, 1+18i),\, (1,\, 1,\, 18n-1)
\]
are removed(Figure 15).
We call these new polycubes the \textbf{3-filler} and the \textbf{3-brick}.

\begin{figure}[ht]
\centering
\includegraphics[width=0.8 \textwidth]{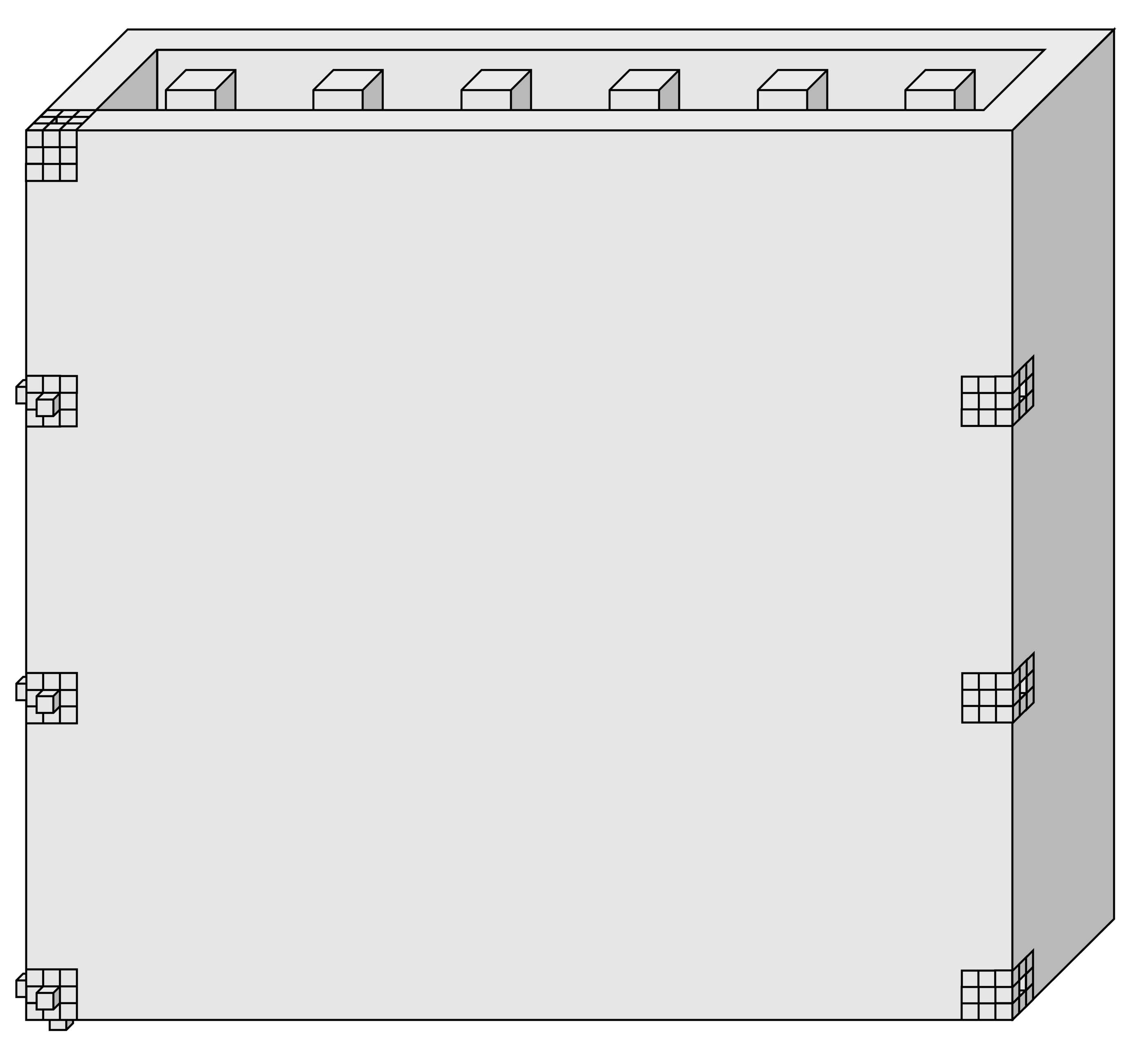}
\captionsetup{width=1 \textwidth}
\caption{The $3$-brick; the figure uses $n=3$ and $m=6$, but the actual $3$-brick would have to satisfy $\gcd(6, \,n)=1$. The towers have been ommited.}
\label{fig:Figure15}
\end{figure}

\begin{lemma} 
The 3-filler and the 3-brick tile $\mathbb{Z}^3$ if and only if $\cS$-cyclic triomino problem is solvable.
\end{lemma}

\begin{proof}
Suppose that the 3-filler and the 3-brick tile $\mathbb{Z}^3$.
The 3-filler cannot tile $\mathbb{Z}^3$ by Lemma 4.2, so the 3-brick must be used.  
The dents and bumps on the top and bottom of the 3-brick force the 3-bricks to form an infinite vertical array, with a period of $6n$.
The bumps and dents on the side of the arrays force the edges of these arrays to be aligned,
and force the $z$ coordinates of adjacent 3-bricks to differ by a multiple of $6$.  
This satisfies the condition of Lemma 4.3, so $\cS$-cyclic triomino problem is solvable.

Now suppose that $\cS$-cyclic triomino problem is solvable.  
It is easy to verify that a tiling by the 3-filler and the 3-brick exists in the form of the above configuration.  
Therefore, a tiling exists if and only if $\cS$-cyclic triomino problem is solvable.
\end{proof}

By Lemma 4.4, it is undecidable whether a set of two disconnected polycubes can tile $\mathbb{Z}^3$.  
Since any set of disconnected polycubes can be simulated by the same number of connected polycubes, this proves Theorem 4.1.

\section{Extension of Theorem 2.5 to higher dimensions}

This section generalizes the result of Theorem 2.5 to dimensions $3$ and above.

\begin{theorem} 
Let $n \geq 3$ be a fixed integer. Given a set of $k$ $n$-dimensional tiles, some of which may be disconnected,
there is a set of $k$ connected $n$-dimensional tiles such that the set of tilings by the two sets of tiles have a one-to-one correspondence.
\end{theorem}

Let $m$ be a positive integer.
We will partition $\{0,\,1,\,\dots,\,m+1\}^n$, which is the $n$-dimensional cube with side length $m+2$.
We use the $3$-ary function $f$ defined in Section 2 to define the $n$-ary function $f_n$.
Let
\[
f_3(x_1,\,x_2,\,x_3) = f(x_1,\,x_2,\,x_3)
\]
and for $n \geq 4$:
\[
f_n(x_1,\,x_2,\,x_3,\,\dots,\,x_n) =
\begin{cases}
f_{n-1}(x_1,\,x_2,\,x_3,\,\dots,\,x_{n-1}) & (0 \leq x_n \leq m) \\
f_{n-1}(m+1-x_2,\,x_1,\,x_3,\,\dots,\,x_{n-1}) & (x_n = m+1)
\end{cases}
\]

$f_n$ can be thought of as taking $m+2$ copies of $f_{n-1}$ and rotating the top(`top' here means having the largest $n$-th dimensional coordinate)
copy by $-\frac{\pi}{2}$ in the $xy$-plane.

Let $Q_{n,\,1},\, Q_{n,\,2},\,\dots,\, Q_{n,\,m}$ be a partition of $\{0,\,1,\,\dots,\,m+1\}^n$ such that:
\[
(x_1,\,x_2,\,\dots,\,x_n) \in Q_{n,\,i} \quad \text{if and only if} \quad f_n(x_1,\,x_2,\,\dots,\,x_n) = i
\]

\begin{lemma} 
$Q_{n,\,1},\, Q_{n,\,2},\,\dots,\,Q_{n,\,m}$ are all connected tiles.
\end{lemma}

\begin{proof}
We will use induction on $n$.
We know that the statement is true for $n=3$ from Section 2.
Suppose that $Q_{n-1,\,1},\,Q_{n-1,\,2},\,\dots,\,Q_{n-1,\,m}$ are all connected tiles.
Let $(x_1,\,x_2,\,\dots,\,x_n)$ represent the $n$-dimensional coordinate.
Since the tiles are obviously connected for $0 \leq x_n \leq m$, we only have to show that the tiles are connected between $x_n = m$ and $x_n = m+1$.
For $1 \leq i \leq m$, we have
\[
(i,\,i,\,0,\,0,\,\dots,\,0,\,m) \in Q_{n,\,i}
\]
\[
(i,\,i,\,0,\,0,\,\dots,\,0,\,m+1) \in Q_{n,\,i}
\]
so the $i$-th tile is connected between $x_n = m$ and $x_n = m+1$.
\end{proof}

\begin{lemma} 
The partition $Q_{n,\,1},\,Q_{n,\,2},\,\dots,\,Q_{n,\,m}$ is internally adjacent.
\end{lemma}

\begin{proof}
The parts $Q_{n,\,1},\,Q_{n,\,2},\,\dots,\,Q_{n,\,m}$ each contain $Q_{3,\,1},\,Q_{3,\,2},\,\dots,\,Q_{3,\,m}$, which is internally adjacent.
Therefore, $Q_{n,\,1},\,Q_{n,\,2},\,\dots,\,Q_{n,\,m}$ is also internally adjacent.
\end{proof}

\begin{lemma} 
The partition $Q_{n,\,1},\,Q_{n,\,2},\,\dots,\,Q_{n,\,m}$ is externally adjacent.
\end{lemma}

\begin{proof}
We will use induction on $n$.
We know that the statement is true for $n=3$ from Section 2.
Suppose that $Q_{n-1,\,1},\,Q_{n-1,\,2},\,\dots,\,Q_{n-1,\,m}$ is externally adjacent.
Since $Q_{n,\,1},\,Q_{n,\,2},\,\dots,\,Q_{n,\,m}$ contains $Q_{n-1,\,1},\,Q_{n-1,\,2},\,\dots,\,Q_{n-1,\,m}$,
the condition for the externally adjacent property holds for every translation within the first $n-1$ dimensions.
Thus, it is sufficient to show that for $1 \leq i < j \leq m$ and the $n$-dimensional vector ${v} = (0,\,0,\,\dots,\,0,\,\pm 1)$,
there exist $a \in Q_{n,\,i}$ and $b \in Q_{n,\,j}$ such that $a + (m+1)v = b$.
In fact:
\[
(j,\,i,\,0,\,0,\,\dots,\,0,\,0) \in Q_{n,\,i}, \quad (j,\,i,\,0,\,0,\,\dots,\,0,\,m+1) \in Q_{n,\,j}
\]
\[
(i,\,j,\,0,\,0,\,\dots,\,0,\,m+1) \in Q_{n,\,i}, \quad (i,\,j,\,0,\,0,\,\dots,\,0,\,0) \in Q_{n,\,j}
\]
Therefore, $Q_{n,\,1},\,Q_{n,\,2},\,\dots,\,Q_{n,\,m}$ is externally adjacent.
\end{proof}

Following the process in Section 2, we now scale up $Q_{n,\,1},\,Q_{n,\,2},\,\dots,\,Q_{n,\,m}$ by a factor of $3$ and add bumps and dents to the new $Q_{n,\,1}$.
The $n$ coordinates
\[
\{(-1,\,1,\,1,\,\dots,\,1),\,(1,\,-1,\,1,\,\dots,\,1),\,\dots,\,(1,\,1,\,1,\,\dots,\,-1)\}
\]
are added, and the $n$ coordinates
\[
\{(3m+5,\,1,\,1,\,\dots,\,1),\,(1,\,3m+5,\,1,\,\dots,\,1),\,\dots,\,(1,\,1,\,1,\,\dots,\,3m+5)\}
\]
are removed.
We define this new set of $n$-dimensional tiles to be $Q_{n,\,1}'$, $Q_{n,\,2}'$, $\dots$, $Q_{n,\,m}'$.

For a positive integer $l$, let $m = l^n - (l - 2)^n$ and $Q_{n,\,1}',\,Q_{n,\,2}',\,\dots,\,Q_{n,\,m}'$ be the corresponding sets from the construction above.
Let $x_{n,\,1},\,x_{n,\,2},\,\dots,\,x_{n,\,m}$ be all the elements of $\{0,\,1,\,\dots,\,l-1\}^n \setminus \{1,\,2,\,\dots,\,l-2\}^n$.
We define $S_{n,\,l}$ as:
\[
S_{n,\,l} = \bigcup_{i=1}^m \left( Q_{n,\,i}' + (3m+6) x_{n,\,i} \right)
\]

From this point, we can replicate the exact arguments of Section 2.2 with the tile $S_{n,\,l}$.
As a result, we obtain Theorem 5.1.

\section{Conclusion} 

In this paper, we prove the undecidability of the translational tiling problem in $\mathbb{Z}^3$ with a set of two connected polycubes.
In the process, we propose the idea of simulating disconnected tiles with connected tiles.
Theorem 5.1 states that in dimensions $3$ and above, tiling with disconnected tiles is fundamentally equivalent to tiling with connected tiles.
Unlike many construction techniques that require two or more tiles to interact, the result of Theorem 5.1 can be directly used for constructing a single tile.
Therefore, this is an important contribution towards proving the undecidability of the translational monotiling problem.

It is also worth noting that we do not directly encode a set of Wang tiles in proving Theorem 4.1.
The arguments for analyzing tiling possibilities could be made much simpler than other papers on this topic due to a rather straightforward encoding.
This is made possible by encoding the cyclic triomino problem, which does not require an absolute coordinate system to encode the states.

The translational tiling problem still remains open for $n \ge 3, \, k=1$ and $(n,\,k)=(2,\,2),\,(2,\,3)$.
It will be interesting to investigate these parameters for decidability or undecidability.

\end{document}